\documentclass[11pt]{amsart}
\usepackage{categorytheory}
\fullpage

\newcommand{\bs}{\smallsetminus}
\newcommand{\V}{\mathcal{V}}

\newcommand{\LL}{\mathbb{L}}
\newcommand{\Asb}{\mathbf{Asm}}
\newcommand{\A}{\mathbb{A}}
\renewcommand{\P}{\mathbb{P}}
\newcommand{\ab}{{\mathrm{ab}}}
\newcommand{\defeq}{\stackrel{\mathrm{def}}{=}}

\renewcommand{\dot}{{{\raisebox{.2ex}{\scalebox{.3}{$\bullet$}}}}}
\DeclareMathOperator{\hocofib}{cofib}

\def\longinlinearrowlength{.7}
\newcommand{\rdash}{\inlineArrow{->,densely dashed}}
\newcommand{\qqandqq}{\qquad\hbox{and}\qquad}

\newtheorem{question}{Question}
\newtheorem*{conjecture}{Conjecture}
\newtheorem{maintheorem}{Theorem}

\newtheorem*{theorem:zero=>notinj}{Theorem \ref{thm:zero=>notinj}}
\newtheorem*{theorem:kernelL}{Theorem \ref{thm:kernelL}}
\newtheorem*{theorem:ll}{Theorem \ref{thm:ll}}
\newtheorem*{theorem:graded}{Theorem \ref{thm:graded}}

\theoremstyle{remark}
\newtheorem{example}[equation]{Example}

\begin{document}

\title{The annihilator of the Lefschetz motive}
\author{Inna Zakharevich}
\begin{abstract}
  In this paper we study a spectrum $K(\V_k)$ such that $\pi_0 K(\V_k)$ is the
  Grothendieck ring of varieties and such that the higher homotopy groups
  contain more geometric information about the geometry of varieties.  We use
  the topology of this spectrum to analyze the structure of $K_0[\V_k]$ and show
  that classes in the kernel of multiplication by $[\A^1]$ can always be
  represented as $[X] - [Y]$, where $[X] \neq [Y]$, $X\times \A^1$ and $Y\times
  \A^1$ are not piecewise isomorphic, but $[X\times \mathbb{A}^1] =[Y \times
  \mathbb{A}^1]$ in $K_0[\V_k]$.  Along the way we present a new proof of the
  result of Larsen--Lunts on the structure on $K_0[\V_k]/([\A^1])$.
\end{abstract}

\maketitle

The Grothendieck ring of varieties over a field $k$---denoted $K_0[\V_k]$---is
defined to be the free abelian group generated by varieties over $k$, modulo the
relation
\[[X] = [Y] + [X\bs Y] \qquad \text{for $Y$ a closed subvariety of $X$}.\] 
The ring structure is given by the formula $[X][Y] = [X\times Y]$.
This ring was first introduced by Grothendieck in a letter to Serre in 1964, and has
since appeared in various places in the study of motivic integration and
birational geometry.  (For more on the Grothendieck ring, see for example
\cite{larsenlunts03}, \cite{lamysebag}, \cite{liusebag}.)  This ring is quite
complicated; for example, it is not an integral domain
(\cite{poonen02}).  

There are two important structural questions about $K_0[\V_k]$.  The
Grothendieck ring of varieties comes with a filtration on the generators given
by the dimension of the variety, where the $n$-th graded piece of the filtration
is given by those elements that can be represented as a formal sum of varieties
of dimension at most $n$.  Equivalently, the $n$-th graded piece is the image of
the homomorphism
\[\psi_n:\Z[X \,|\, \dim X \leq n]/([X] = [Y] + [X\bs Y]) \rto K_0[\V_k].\]

In \cite[Question 1.2]{larsenlunts03} (and, equivalently, in
\cite[p121]{gromov99}) the following question is asked:
\begin{question}\lbl{q:ll}
  Is $\psi_n$ injective?  Equivalently, do all birational automorphisms of
  varieties extend to piecewise isomorphisms?
\end{question}
By a piecewise isomorphism between varieties $X$ and $Y$ we mean a pair of
stratifications of $X$ and $Y$ and isomorphisms of corresponding strata.  By the
definition of $K_0[\V_k]$, if $X$ and $Y$ are piecewise isomorphic then $[X] =
[Y]$ in $K_0[\V_k]$; the question is interested in the converse of this
statement.  In recent work of Borisov \cite{borisov14} and Karzhemanov
\cite{karzhemanov14} counterexamples to the injectivity of $\psi_n$ are
constructed when $k$ is a subfield of $\mathbf{C}$, and we expect that in fact
it will not be injective over any $k$.  The associated graded of the filtration
induced on $K_0[\V_k]$ by the images of $\psi_n$ is still unknown.

The second structural question concerns multiplication by the Lefschetz motive
$\LL \defeq [\A^1]$.  In Kontsevich's definition of the motivic integral, the
motivic measure has values in $K_0[\V_k]$, but is only well-defined up to a
power of $\LL$.  Thus in order for motivic integration to be well-defined it is
necessary to invert $\LL$ and consider the ring $K_0[\V_k][\LL^{-1}]$.  (For more
on motivic integration, see for example \cite{denefloeser99}.)  This leads to
the following question.
\begin{question} \lbl{q:L}
  Is $\LL$ a zero divisor?
\end{question}
This question was also recently resolved by Borisov in \cite{borisov14}.  In
fact, Borisov's main result was to construct an element in the kernel of
multiplication by $\LL$, and, seemingly coincidentally, his method also
constructed an element in the kernel of $\psi_n$.  A precise algebraic
description of the kernel of multiplication by $\LL$, and, more interestingly, of
the kernel of the localization $K_0[\V_k] \rto K_0[\V_k][\LL^{-1}]$, is still
unknown.

In this paper we propose to replace the Grothendieck ring of varieties with a
spectrum $K(\V_k)$, which we call the \textsl{Grothendieck spectrum of
  varieties}.\footnote{For those readers unfamiliar with stable homotopy, almost
  all results in the paper still hold true if ``spectrum'' is replaced
  everywhere with ``homology theory.''  In particular, all homotopy groups of a
  spectrum are abelian groups.}  This spectrum has the property that
$\pi_0K(\V_k) \cong K_0[\V_k]$, and that its higher homotopy groups contain
further geometric information about the decomposition of varieties.  This
spectrum encodes additional geometric information which is lost when passing to
the Grothendieck ring.  We use topological tools to study
the structure of this spectrum and illuminate a precise relationship between 
Questions \ref{q:ll} and \ref{q:L}.

The spectrum $K(\V_k)$ also comes with a dimension filtration which comes from
the dimension of the varieties which ``generate'' the spectrum.  Although the
associated graded of $K_0[\V_k]$ is unknown, it turns out that the associated
graded of this spectrum has a beautifully simple description.
\begin{maintheorem} \lbl{thm:graded} The category $\V_k$ of varieties over $k$ and
  locally closed inclusions comes with a filtration, where the $n$-th filtered
  piece $\V_k^{(n)}$ is generated by varieties of dimension $n$.  This
  filtration gives rise to a filtration on $K(\V_k)$, whose $n$-th graded piece
  is
  \[\bigvee_{[X]\in B_n} \left(\Sigma^\infty B\Aut_k k(X)_+\right),\]
  where $B_n$ is the set of birational isomorphism classes of varieties of
  dimension $n$ and $\Aut_k k(X)$ is the group of birational automorphisms of
  the variety $X$.
\end{maintheorem}

This theorem implies that there exists a spectral sequence in stable homotopy
\[E^1_{p,q} = \bigoplus_{[X]\in B_q} \pi_p^s B\Aut_k k(X) \oplus
\pi_p^s S^0 \Rto \pi_p K(\V_k);\] in particular, the $p=0$ column converges to
$K_0[\V_k]$.\footnote{This grading is inspired by the grading for the Adams
  Spectral Sequence, and is optimized so that our spectral sequence is a
  first-quadrant spectral sequence and fits comfortably on the page.  See
  Figure~\ref{fig:ss-kv} on page \pageref{fig:ss-kv}.}  By
analyzing the construction of $K(\V_k)$ we can compute the differentials between
the $1$-st and $0$-th columns in the spectral sequence to get the following:
\begin{maintheorem} \lbl{thm:diff=ker}
  There exist nonzero differentials between the $1$-st and $0$-th column of
  some page of the spectral sequence if and only if $\psi_n$ has a nonzero
  kernel for some $n$.
\end{maintheorem}

This theorem holds over all fields $k$.  To prove this we show that the spectral
sequence constructed above produces an obstruction theory for birational
automorphisms extending to piecewise isomorphisms, therefore answering the
question of Gromov in \cite[p121]{gromov99}; this is proved in more detail in
Theorem~\ref{thm:ll<=>diffs}.  At first it may seem that this does not give very
much information about $\psi_n$, since the birational automorphism groups of
varieties are very complex.  However, this repackaging allows us to see a clear
connection between the answers to Questions \ref{q:ll} and \ref{q:L}.  In order
to make these conclusions, though, we need to restrict to ``convenient'' fields
$k$.  At the moment, only fields with characteristic $0$ are known to be
convenient.  For more details, see Definition~\ref{def:convenient}.
\begin{maintheorem} \lbl{thm:zero=>notinj} Suppose that $k$ is a convenient
  field.  If $\LL$ is a zero divisor in $K_0[\V_k]$ then $\psi_n$ is not
  injective for some $n$.
\end{maintheorem}

Moreover, we can say something about the kernel of multiplication by $\LL$:
\begin{maintheorem} \lbl{thm:kernelL} Suppose that $k$ is a convenient field.
  If $\chi$ is in the kernel of multiplication by $\LL$ then we can represent
  $\chi$ as $[X]-[Y]$ where $[X \times \A^1] = [Y \times \A^1]$ but $X\times
  \A^1$ and $Y\times \A^1$ are not piecewise isomorphic.
\end{maintheorem}
These theorems tell us that Borisov's beautiful coincidence was not a
coincidence at all: any element in the kernel of multiplication by $\LL$ produces
an element in the kernel of $\psi_n$.

The proofs of Theorems \ref{thm:zero=>notinj} and \ref{thm:kernelL} are based on
the following observation.  We have a map of spectra $\cdot \LL: K(\V_k) \rto
K(\V_k)$ induced by multiplication by $\LL$.  If we can construct the cofiber $C$
of this map then we get an exact sequence
\[K_1(\V_k) \rto^{pr} \pi_1C \rto \pi_0K(\V_k) \rto^{\cdot \LL} \pi_0K(\V_k) \rto
\pi_0C.\] The map $\cdot \LL$ is exactly multiplication by $\LL$ on $K_0[\V_k]$.
Since this sequence is exact, the kernel of this map is the same as the cokernel
of the map $pr$, and it turns out that this cokernel is much simpler to study
than the kernel of $\LL$.  This approach also produces an alternate proof of the
following result, which was originally proved in \cite[Theorem
2.3]{larsenlunts03}:
\begin{maintheorem} \lbl{thm:ll} There is an isomorphism of abelian groups
  $K_0[\V_{\mathbf{C}}]/(\LL) \rto \Z[SB]$, where $SB$ is the set of stable
  birational isomorphism classes of varieties.
\end{maintheorem}

\begin{remark}
  The theorem of Larsen and Lunts actually claims that this is an isomorphism of
  rings.  The techniques explained in this paper should prove that statement as
  well, since the spectral sequence comes with a multiplicative structure.  The
  details of the multiplicative structure are deferred until future work.
\end{remark}

As an amusing corollary of our techniques, we also give an extension of Liu and
Sebag's result \cite[Corollary 5(1)]{liusebag} that if $X$ and $Y$ are smooth
projective varieties and $[X] = [Y]$ in $K_0[\V_k]$ then $X$ and $Y$ are stably
birational; this is proved in Corollary~\ref{cor:ls}.  In our extension we show
that if two varieties are the smallest-dimensional representatives of their
stable birational isomorphism classes then if they are equal in the Grothendieck
ring then they are stably birational, regardless of whether they are smooth or
projective.  Note that this is not simply a corollary of Theorem~\ref{thm:ll},
as the homomorphism in Theorem~\ref{thm:ll} only takes smooth projective
varieties to their stable birational isomorphism classes; the image of the
homomorphism on a general variety is not given simply by its stable birational
isomorphism class.

This paper is organized as follows.  In Section~\ref{sec:assemblers} we
introduce the technical machinery of assemblers and their $K$-theory, which are
the tools necessary to define $K(\V_k)$ and prove Theorem~\ref{thm:graded}.  In
Section~\ref{sec:ss} we give a brief review of the facts about spectral
sequences that we need.  Section~\ref{sec:1stss} concludes the analysis of
the filtration on $K(\V_k)$ and proves Theorem~\ref{thm:diff=ker}.
Section~\ref{sec:2ndss} analyzes the cofiber of multiplication by $\LL$
in general fields.  Finally, Section~\ref{sec:convenient} restricts to
convenient fields and proves Theorems \ref{thm:zero=>notinj}, \ref{thm:kernelL}
and \ref{thm:ll}.

\subsection*{Acknowledgements} The author would like to thank Dan Grayson, Mike
Hopkins, Peter May, Matthew Morrow, Madhav Nori, Aaron Silberstein and Chuck
Weibel for helpful discussions of this material, and Mike Hill for a very
illuminating discussion during the JMM.  The author would also like to thank the
Institute for Advanced Study and the University of Chicago, as well as the NSF
MSRFP grant for financial support.

\subsection*{Notation} We write $\Sp$ for the category of spectra.  In this
paper, we take as our model for this category the category of symmetric spectra
of simplicial sets.  When a model structure is required, we take the stable
structure with the levelwise cofibrations; for more detail, see
Remark~\ref{rem:inc-cofib}.  All homotopy groups are stable homotopy groups; in
forthcoming sections we drop the $s$ superscript.  For a topological space
$X$ we write $\Sigma_+^\infty X$ for the spectrum $\Sigma^\infty(X_+)$.

For any set $S$, we write
$\Z^{\oplus S} \defeq \bigoplus_{s\in S} \Z$.  We also use the highly
nonstandard notation
\[\tilde\Z^{\oplus S} \defeq
\begin{cases}
  \ker (\Z^{\oplus S} \rto^+ \Z) & \hbox{if } S \neq \eset,
  \\
  \Z/2 & \hbox{if } S = \eset.
\end{cases}\] We write $\uplus$ to refer to a union which is disjoint as sets.

\section{Introducing assemblers} \lbl{sec:assemblers}

Assemblers are introduced in \cite{Z-Kth-ass}.  In this paper we quote the
definition of an assembler and as few theorems as is necessary to facilitate
understanding of the main results; for a more detailed exposition, see
\cite{Z-Kth-ass}.

\begin{definition}
  An \textsl{assembler} is a small Grothendieck site which satisfies the
  following extra axioms.
  \begin{itemize}
  \item[(I)] $\C$ has an initial object $\initial$ and the empty family is a
    covering family of $\initial$.
  \item[(R)] For any object $A$ in $\C$, any two finite disjoint covering
    families of $A$ have a common refinement which is itself a finite disjoint
    covering family.
  \item[(M)] All morphisms in $\C$ are monomorphisms.
  \end{itemize}

  We say that a family of morphisms $\{f_i:A_i \rto A\}_{i\in I}$ in an
  assembler $\C$ is \textsl{disjoint} if for all $i\neq j\in I$, the pullback of
  \[A_i \rto^{f_i} A \lto^{f_j} A_j\]
  exists and is equal to $\initial$.  When $\C$ has all pullbacks, (R) holds
  automatically.
  
  The category of assemblers, denoted $\Asb$, has
  \begin{description}
  \item[objects] assemblers $\C$, and
  \item[morphisms] functors $\C \rto \D$ which are continuous with respect to
    the Grothendieck topology, preserve the initial object, and which take
    disjoint families to disjoint families.
  \end{description}
\end{definition}

The fundamental result about the category of assemblers is the following.

\begin{theorem} \lbl{thm:mainass}
  There is a functor $K:\Asb \rto \Sp$ with the following properties:
  \begin{enumerate}
  \item The group $\pi_0 K(\C)$ is the free abelian group on objects of $\C$,
    under the relations that for every finite disjoint covering family $\{f_i:
    A_i \rto A\}_{i\in I}$ in $\C$,
    \[[A] = \sum_{i\in I} [A_i].\]
  \item Every element of $\pi_1K(\C)$ can be represented by the following data:
    \begin{itemize}
    \item a pair of finite tuples $\{A_i\}_{i\in I}$, $\{B_j\}_{j\in J}$ of
      objects in $\C$, and
    \item for $\epsilon=\pm 1$, a map of finite sets $f_\epsilon: I \rto J$
      and for all $i\in I$, morphisms $f_{\epsilon i}:A_i \rto
      B_{f_\epsilon(i)}$ such that for each $j\in J$ the family
      $\{f_{\epsilon i}: A_i \rto B_j\}_{i\in f_\epsilon^{-1}(j)}$
      is a covering family.
    \end{itemize}
  \end{enumerate}
\end{theorem}

For conciseness we write $K_i(\C)$ for $\pi_i K(\C)$.

We purposefully leave the relations between generators in part (2) of the
theorem imprecise, as this is the most general statement we need in this paper.
For a more precise statement see \cite[Corollary 3.10]{Z-ass-pi1}.

The main example of an assembler that we are interested in in this paper is the
assembler of varieties.

\begin{example}
  Let $k$ be a field.  The assembler $\V_k$ has as objects varieties over $k$;
  here by ``variety'' we mean a reduced separated scheme of finite type over
  $k$.  The morphisms in $\V_k$ are locally closed embeddings.  The Grothendieck
  topology is generated by the coverage $\{Y \rcofib X, X\bs Y\rcofib X\}$ for
  $Y$ a closed subvariety of $X$.  (For background on coverages, see for example
  \cite{johnstone02v2}.)

  When $k$ is clear from context we write $\V$ instead of $\V_k$.
\end{example}

\begin{example}
  Let $\FinSet$ be the category of finite sets and injective morphisms.  Then
  this is an assembler if we define a family of morphisms $\{f_i:A_i \rto
  A\}_{i\in I}$ to be a covering family if $\bigcup_{i\in I} f_i(A_i) = A$.
  Then $K_0(\FinSet) \cong \Z$, with the set $\{1,\ldots,n\}$ representing $n\in
  \Z$.

  If $k$ is a finite field then the functor $X\rgoesto X(k)$ gives a morphism of
  assemblers $\V \rto \FinSet$, and thus a map of spectra $K(\V) \rto
  K(\FinSet)$.  On $K_0$ this is just point counting over $k$.  By the
  Barratt--Priddy--Quillen Theorem, $K(\FinSet) \simeq \S$, the sphere spectrum,
  so we see that point counting lifts to a map of spectra
  $K(\V) \rto \S$.
\end{example}

There are two other important types of assemblers that arise in our
analysis of $K(\V)$.

\begin{example}
  Let $G$ be a discrete group.  The assembler $\S_G$ has two objects $\initial$
  and $*$, with one morphism $\initial \rto *$ and $\Aut(*) \cong G$.  Then
  $K(\S_G) \simeq \Sigma_+^\infty BG$; in particular, when $G$ is trivial $\S_G$
  is weakly equivalent to the sphere spectrum.
\end{example}

\begin{example}
  Suppose that $\C$ and $\D$ are assemblers.  Then we can construct an assembler
  $\C \vee \D$ whose underlying category is the union of the categories $\C$ and
  $\D$ with the two initial objects glued together.  The Grothendieck topology
  on $\C\vee\D$ is inherited from the topologies on $\C$ and $\D$.
\end{example}

The assembler $\V$ comes with an associated filtration induced by the
dimension of the corresponding varieties.  More concretely, let $\V^{(n)}$ be
the full subcategory of varieties of dimension at most $n$; it has the structure
of an assembler induced from that of $\V$, and the inclusion is a morphism of
assemblers.  Thus we get a diagram
\[\V^{(0)} \rto \V^{(1)} \rto \V^{(2)} \rto \cdots \rto
V^{(n)} \rto \cdots \rto \V.\]
Applying $K$ to this sequence, we get a sequence of morphisms of spectra
\[K(\V^{(0)}) \rto K(\V^{(1)}) \rto K(\V^{(2)}) \rto \cdots \rto
K(V^{(n)}) \rto \cdots \rto K(\V),\] such that $K(\V)$ is the colimit of
this diagram; thus $K(\V)$ is a filtered spectrum.

\begin{remark} \lbl{rem:inc-cofib} In order for some of the later techniques
  used in the paper to work, we need the morphisms of spectra to be cofibrations
  in some stable model structure on $\Sp$.  In order to accomplish this, we take
  as our model of the stable homotopy category the category of symmetric spectra
  of simplicial sets with the stable model structure.  (See
  \cite[Theorem9.2]{mmss}.)  In this structure, the weak equivalences are the
  stable equivalences and the cofibrations are the level cofibrations.  The
  construction of $K$ defined in \cite[Definition 1.9]{Z-Kth-ass} gives the
  desired properties.
\end{remark}

We now want to compute the spectral sequence induced by this filtered spectrum.
In order to do this, we must first compute $\pi_q(K(\V)^{(p)}, K(\V)^{(p-1)})$
for all $q$ and $p$.  We will not give the full details of the proof; instead,
we explain how to use the machinery developed in \cite{Z-Kth-ass} to obtain
these results.

We write $\V^{(n,n-1)}$ for the assembler whose underlying category is the
full subcategory of $\V$ consisting of varieties of dimension exactly $n$ and
the empty variety.  For a variety $X$, a family $\{f_i:X_i \rto X\}$ is a
covering family if it can be completed to a disjoint covering family inside
$\V^{(n)}$ by morphisms whose domains have dimension at most $n-1$.  Thus for an
irreducible variety $X$, a family $\{f_i:X_i \rto X\}_{i\in I}$ is a covering
family if and only if $X_i \neq \initial$ for a unique $i\in I$.

\begin{theorem} \lbl{thm:cofib}
  The cofiber of the map $K(\V^{(n-1)}) \rcofib K(\V^{(n)})$ is
  $K(\V^{(n,n-1)})$.  
\end{theorem}

This is a special case of \cite[Theorem D]{Z-Kth-ass}.  Computing
the homotopy type of $K(\V^{(n,n-1)})$ require the use of an approximation
theorem for assemblers:

\begin{theorem}[{\cite[Theorem B]{Z-Kth-ass}}] \lbl{thm:1stapprox} Suppose
  that $\D$ is a subassembler of $\C$ such that every object $A$ of $\C$ has a
  finite disjoint covering family $\{f_i:A_i \rto A\}_{i\in I}$ such that $A_i$
  is in  $\D$ for all $i$.  Then the map $K(\D) \rto K(\C)$ induced by the inclusion
  $\D \rcofib \C$ is a homotopy equivalence of spectra.
\end{theorem}

This is a devissage result for assemblers: if the objects in an assembler can be
covered by the objects in a subassembler, then the inclusion of assemblers
induces a homotopy equivalence on $K$-theory.  For example, let $\V^{sm}$ be the
full subcategory of $\V$ consisting of the smooth varieties.  Then the inclusion
$\V^{sm} \rto \V$ is an inclusion that satisfies the conditions of the
proposition, since any variety can be stratified by smooth varieties, and thus
the inclusion
$K(\V^{sm}) \rto K(\V)$
is a homotopy equivalence of spectra.

\begin{definition} \lbl{def:Bn}
  Let $B_n$ be the set of birational isomorphism classes of irreducible
  varieties over $k$ of dimension $n$.  For any $\alpha$ in $B_n$, define
  \[\Aut(\alpha) = \mathrm{Aut}_k k(X)\]
  for any variety $X$ in the birational isomorphism class $\alpha$.
\end{definition}

Theorem~\ref{thm:graded} follows directly from several results in
\cite{Z-Kth-ass}. Here, we give an outline of the proof by reducing of the
theorem to those results.


\begin{proof}[Proof of Theorem~\ref{thm:graded}]
  By Theorem~\ref{thm:cofib}, it suffices to show that for any $n \geq 0$,
  \[K(\V^{(n,n-1)}) \simeq \bigvee_{\alpha\in B_n} \Sigma^\infty_+ B\Aut(\alpha).\]
  Let $\tilde \V^{(n,n-1)}$ be the full subassembler of $\V^{(n,n-1)}$ of
  irreducible varieties.  By Theorem~\ref{thm:1stapprox}, 
  \[K(\tilde \V^{(n,n-1)}) \simeq K(\V^{(n,n-1)}).\]

  Pick a representative $X_\alpha$ for each equivalence class $\alpha$ in $B_n$,
  and let $\C$ be the full subassembler of $\tilde \V^{(n,n-1)}$ consisting of
  subvarieties of $X_\alpha$ for any $\alpha$.  By Theorem~\ref{thm:1stapprox},
  $K(\C) \simeq K(\V^{(n,n-1)})$.  Note that each nonempty variety in $\C$ has a
  morphism to exactly one of the $X_\alpha$, since all morphisms are inclusions
  of dense open subsets and if some $Z$ had morphisms to $X_\alpha$ and
  $X_\beta$ then $X_\alpha$ and $X_\beta$ are by definition birationally
  isomorphic,  hence $\alpha = \beta$.  We can therefore write
  \[\C \cong \bigvee_{\alpha\in B_n} \C_{X_\alpha},\]
  where $\C_{X_\alpha}$ is the full subassembler of $\C$ consisting of those
  varieties with morphisms to $X_\alpha$.  $K$-theory commutes with $\vee$, so
  it remains to analyze $K(\C_{X_\alpha})$ for all $\alpha$ in $B_n$. This is
  done in \cite[Theorem 2.1(1)]{Z-Kth-ass}, which states exactly that
  $K(\C_{X_\alpha}) = \Sigma_+^\infty B \Aut(X_\alpha)$.
\end{proof}

We also need a construction of the cofiber of a morphism of (simplicial)
assemblers, by which we mean functors $\Delta^{op} \rto \Asb$.  In general, the
cofiber of a morphism of assemblers (or a morphism of simplicial assemblers) is
a simplicial assembler, so this level of generality is necessary.  The category
of assemblers sits inside the category of simplicial assemblers as the constant
simplicial assemblers. (For more on simplicial objects, see for example
\cite[Chapter 3]{hovey99}; for a more in-depth discussion of this definition and
Theorem~\ref{thm:cofiber} see \cite[Section 4]{Z-Kth-ass}.)
\begin{definition} \lbl{def:simp-cofib}
  For any assembler, $\C$, write $\nabla: \C\vee\C \rto \C$ for the fold map
  which is given by the identity on each of the two wedge summands.

  Let $F:\C_\dot\rto \D_\dot$ be a morphism of simplicial assemblers.  The simplicial
  assembler $(\D_\dot/F)_\dot$ is defined by
  \[(\D_\dot/F)_n = \D_n \vee \bigvee_{i=1}^n \C_n.\]
  The face maps $d_i$ for $n>i>0$ are defined by the composition
  \[\C_n \vee \C_n \rto^{d_i \vee d_i} \C_{n-1} \vee \C_{n-1} \rto^\nabla \C_{n-1},\]
  on the $i$-th and $i+1$-st copies of $\C_n$, and by $d_i$ on the other wedge
  summands.  The map $d_0$ is defined by the composition
  \[\D_n \vee C_n \rto^{d_0\vee d_0} \D_{n-1} \vee \C_{n-1} \rto^{1 \vee F}
  \D_{n-1} \vee \D_{n-1} \rto^\nabla \D_{n-1}\] on the first two wedge summands
  and $d_0$ on the others, and $d_n$ is defined by mapping the $n$-th copy of
  $\C_n$ entirely to the initial object and by $d_n$ elsewhere.

  By abusing notation, we generally write $\D_\dot/F$ for the simplicial
  assembler $(\D_\dot/F)$; when $\D_\dot$ is a constant assembler $\D$ we write
  $\D/F$.

  There is a morphism of simplicial assemblers $p:\D_\dot \rto (\D/F)_\dot$
  given by including $\D_n$ into the $\D_n$ in the $n$-th level.

  In the case when $\D_\dot$ is the assembler which is trivial at each level, we
  write $\Sigma \C_\dot$ for $\D_\dot/F$.
\end{definition}

This construction is analogous to the usual bar construction for a 
cofiber of a map of simplicial objects; for more on bar constructions, see
\cite[Section 4.2]{riehl14}.  

\begin{theorem}[{\cite[Theorem C]{Z-Kth-ass}}] \lbl{thm:cofiber} 
  The sequence
  \[K(\C_\dot) \rto^{K(F)} K(\D_\dot) \rto^{K(p)} K((\D/F)_\dot)\] is a
  cofiber sequence.  The boundary morphism $K_{n}((\D/F)_\dot) \rto
  K_{n-1}(\C_\dot)$ is obtained from the morphism of simplicial assemblers
  $(\D/F)_\dot \rto \Sigma\C_\dot$ defined by sending each copy of $\D_n$ to the
  initial object.
\end{theorem}

\section{An aside on spectral sequences} \lbl{sec:ss}

This section contains a quick introduction to the construction of the spectral
sequence associated to a filtered spectrum, geared towards the results we need
in later sections.  In order to make our results easier to read we use a
nonstandard grading in the spectral sequences.  All of the material in this
section is well-known; for a reference, see \cite{mccleary01}.

Suppose that we have a filtered spectrum, by which we mean a diagram
\[{X_0 \rto^{i_1} X_1 \rto^{i_2} \cdots \rto^{i_{n}} X_n
  \rto^{i_{n+1}} \cdots \rto X},\]
where $X = \colim_n X_n$.  Then for each $n$ we have a long exact sequence
\[{\cdots\rto \pi_m(X_{n-1}) \rto^{\iota} \pi_m(X_n) \rto^{p} \pi_m(X_n,X_{n-1})
  \rto^{\partial} \pi_{m-1}(X_{n-1}) \rto \cdots}.\] Since $X$ is a spectrum,
this sequence is infinite in both directions.  We define
\[A_{p,q} = \pi_p(X_q) \qqandqq E^1_{p,q} = \pi_p(X_q,X_{q-1});\]
from the long exact sequences we then have the following diagram:
\begin{diagram-fixed}
  {     
    E^1_{p,q}   & A_{p-1,q-1} & E^1_{p-1,q-1} \\
    E^1_{p,q-1} & A_{p-1,q-2} & E^1_{p-1,q-2} \\
    E^1_{p,q-2} & A_{p-1,q-3} & E^1_{p-1,q-3} \\
  };
  \to{1-1}{1-2}^\partial \to{1-2}{1-3}^p 
  \to{2-1}{2-2}^\partial \to{2-2}{2-3}^p
  \to{3-1}{3-2}^\partial \to{3-2}{3-3}^p
  \to{3-2}{2-2}_{\iota} \to{2-2}{1-2}_{\iota}
\end{diagram-fixed}
By gluing copies of itself for different $p$ and $q$ this diagram can be
continued infinitely in all directions.  Note that each group $E^1_{p,q}$ and
$A_{p,q}$ appears in it only once, and that the original long exact sequences
appear as stair-steps: for example, the sequence
\begin{diagram-fixed}
  { & A_{p-1,q-1} & E^1_{p-1,q-1} \\
    E^1_{p,q-1} & A_{p-1,q-2} \\};
  \to{2-1}{2-2}^\partial
  \to{2-2}{1-2}^\iota \to{1-2}{1-3}^p
\end{diagram-fixed}
is exact.

The differential $d_1:E^1_{p,q} \rto E^1_{p-1,q-1}$ is defined to be $p\partial$
and $E^2_{p,q}\defeq \ker d_1/\im d_1$.  If $a$ is in $\ker d_1$ then $\partial
a$ must be in $\ker p = \im \iota$; thus there exists $a'$ in $A_{p-1,q-2}$ such
that $\iota(a') = a$.  We define $d_2(a) \defeq p(a')$, and $E^3_{p,q} \defeq
\ker d_2/\im d_2$.  This continues onwards: if $d_2(a) = 0$ then $a'$ is in $\ker p
= \im \iota$ and thus there exists $a''$ in $A_{p-1,q-3}$ such that $\iota(a'') =
a'$; we define $d_3(a) \defeq p(a'')$, and so on.  In general,
\[d_r:E^r_{p,q} \rto E^r_{p-1,q-r}.\]

We note the following lemma as it is useful later:
\begin{lemma} \lbl{lem:partial-0}
  For $x$ in $E^1_{p,q}$, if $\partial x = 0$ then $d_rx = 0$ for all $r \geq 1$.  
\end{lemma}
%
With this indexing we have a spectral sequence
\begin{equation}
  \label{eq:genss}
  E^1_{p,q} \Rto \pi_p(X)
\end{equation}
in which $d_r: E^r_{p,q} \rto E^r_{p-1,q-r}$.
If we write $I_n:X_n \rto X$ then on the
$E^\infty$ page we have 
\[E^\infty_{p,q} = (I_q)_*\pi_p(X_q) / \big( (I_{q-1})_*\pi_p(X_{q-1}) \big).\]
Thus the $p$-th column of the $E^\infty$ page is the associated graded of the
filtration on $\pi_p(X)$ induced by the filtered spectrum.
As a consequence
of \cite[Theorem 6.1]{boardman99} we get the following:
\begin{lemma} \lbl{lem:ss-conv}
  If all of the spectra $X_i$ are connective (all $\pi_i = 0$ for $i<0$) then
  this spectral sequence converges to the homotopy groups of $X$.
\end{lemma}
%

\section{A spectral sequence for $K(\V)$} \lbl{sec:1stss}

The dimension filtration on $\V$ produces a filtered spectrum
\[K(\V^{(0)}) \rcofib K(\V^{(1)}) \rcofib \cdots \rcofib K(\V).\]
We now use the tools of Section~\ref{sec:ss} to write down and analyze the
spectral sequence we get from this filtered spectrum.

The filtration on $K(\V)$ gives us the following spectral sequence:
\begin{equation}
  \label{eq:kvss}
  E^1_{p,q} = K_p(\V^{(q,q-1)}) \Rto K_p(\V).
\end{equation}
The differential $d_r$ is a homomorphism $E^r_{p,q} \rto
E^r_{p-1,q-r}$.  As $E^r_{q-r}$ is $0$ for all $r > q$
Lemma~\ref{lem:ss-conv} applies and the spectral sequence converges.
We know that
\[\pi_0(\Sigma_+^\infty BG) \cong \Z \qqandqq \pi_1(\Sigma_+^\infty BG)
\cong G^\ab \oplus \Z/2,\] and since $\pi_*$ is a homology theory
applying $\pi_*$ changes $\bigvee$ to $\bigoplus$.  Therefore we can
compute the first two columns of the spectral sequence; a picture of
this appears in Figure~\ref{fig:ss-kv}.
\begin{figure}[h]
  \centering
  \begin{tikzpicture}[xscale=8]
    \node (OO) at (0,0) {$\Z^{\oplus B_0}$};
    \node (Odots) at (0,1) {$\vdots$};
    \node (On-1) at (0,2) {$\Z^{\oplus B_{n-1}}$};
    \node (On) at (0,3) {$\Z^{\oplus B_n}$};
    \node (FO) at (1,0) {$\bigoplus_{\alpha\in B_0} \Aut(\alpha)^\ab\oplus \Z/2$};
    \node (Fdots) at (1,1) {$\vdots$};
    \node (Fn-1) at (1,2) {$\bigoplus_{\alpha\in B_{n-1}}\Aut(\alpha)^\ab\oplus \Z/2$};
    \node (Fn) at (1,3) {$\bigoplus_{\alpha\in B_n} \Aut(\alpha)^\ab\oplus \Z/2$};
    
    \node (pi0) at (0,-1) {$\pi_0$};
    \node (pi1) at (1,-1) {$\pi_1$};
    \draw (OO.south west) +(-.1,-.1) to[draw] +(1.5,-.1);
    \draw[->] (Fn.west) -- node[auto,swap] {$d_1$} (On-1.east);
    \draw[->, densely dotted] (Fn.west) -- node[below right] {$d_{n}$} (OO.east);
  \end{tikzpicture}
  \caption{Spectral sequence for $K(\V)$}
  \label{fig:ss-kv}
\end{figure}

We would like to compute the differentials in this spectral sequence.  To do
this, we first compute the boundary map $\partial: K_1(\V^{(q,q-1)}) \rto
K_0(\V^{(q-1)})$ in the long exact sequence for the cofiber sequence
$K(\V^{(q-1)}) \rto K(\V^{(q)}) \rto K(\V^{(q,q-1)})$.

\begin{lemma} \lbl{lem:1st-partial-comp}
  Let $\alpha$ be in $B_q$ and let $\varphi$ be in $\Aut(\alpha)$.  For a representative
  $X$ of $\alpha$, let $\varphi$ be represented by an isomorphism $U \rto V$ of
  dense open subsets of $X$.  Then
  \[\partial[\varphi] = [X\bs V] - [X \bs U].\]
\end{lemma}

\begin{proof}
  From \cite[Corollary 3.10]{Z-ass-pi1} we know that each $x$ in
  $K_1(\V^{(q,q-1)})$ can be represented by elements consisting of the following
  data: two finite disjoint covering families
  \[\{f_i: X_i \rto X\}_{i\in I} \cup \{f:Y \rto X\}\] and \[\{g_i:X_i \rto
  X\}_{i\in I} \cup \{g':Z\rto X\},\] such that $\dim Y, \dim Z \leq q-1$.
  More informally, an element of $K_1(\V^{(q,q-1)})$ consists of a
  variety $X$, a dense open subset (represented by the union of the
  $X_i$'s) which is embedded into $X$ in two different ways
  (represented by the maps $f_i$ and $g_i$) as well as the data of the
  two different complements to the embeddings in $X$ ($Y$ and $Z$).
  Applying \cite[Proposition 3.13]{Z-ass-pi1} to this representation
  we get that 
  \[\partial x = [Z]-[Y]\]
  in $K_0(\V^{(q-1)})$.

  Now consider $\varphi$.  The class of $\varphi$ in $K(\V^{(q,q-1)})$ can be
  represented by the two covering families
  \[\{U \rto X\} \cup \{X\bs U \rto X\}
  \quad\hbox{and}\quad \{U \rto^{\varphi} V \rto X\} \cup \{X\bs V
  \rto X\}.\] Therefore
  $\partial[\varphi] = [X \bs V] - [X \bs U]$.
\end{proof}

Thus the boundary map in the long exact sequence associated to the inclusion of
one filtration degree into the next measures the error of a birational
automorphism of the variety extending to a piecewise automorphism.  In order to
make this connection more precise, we start with the following lemma.

\begin{lemma} \lbl{lem:ll<=>diffs} Let $X$ and $Y$ be varieties of dimension at
  most $n$ such that $[X] = [Y]$ in $K_0(\V^{(n,n-1)})$.  Then we can write $X =
  U \uplus X'$ and $Y = V \uplus Y'$ such that $U$ and $V$ are piecewise isomorphic
  and $\dim X', \dim Y' < n$.
\end{lemma}

\begin{proof}
  Write $X = \bigcup_{i=0}^n X_i$, where for $i>0$ the $X_i$ are the irreducible
  components of $X$ of dimension $n$, and $X_0$ is the union of all of the
  irreducible components of $X$ which have dimension less than $n$.  For $i >
  0$, let $U_i = X_i \bs \bigcup_{j \neq i} X_j$, and let $X_0' = X \bs
  \bigcup_{i=1}^n U_i$.  Thus we have
  $X = X_0' \uplus \biguplus_{i=1}^n U_i$
  where $\dim X_0' < n$, $\dim U_i = n$
  and each $U_i$ is irreducible.  Analogously, $Y$ can be written as $Y_0' \uplus
  \biguplus_{j=1}^m V_j$.

  Since all varieties of dimension less than $n$ are $0$ in $K_0(\V^{(n,n-1)})$, 
  $[X] = \sum_{i=1}^n [U_i]$ and $[Y] = \sum_{j=1}^m [V_j]$.  Since every open
  subset of an irreducible variety is dense and $[X]=[Y]$, it follows that $m =
  n$, and we must have a permutation $\sigma$ in $\Sigma_n$ such that $[U_i] =
  [V_{\sigma(i)}]$ in $K_0(\V^{(n,n-1)})$.  Thus for each $i\geq 1$ there exists
  a birational isomorphism $U_i \rdash V_{\sigma(i)}$, given as an isomorphism
  $\varphi_i:U_i' \rto V_{\sigma(i)}'$.  Let $X_i' = U_i \bs U_i'$ and $Y_i' =
  V_i \bs V_i'$; note that $\dim X_i',\dim Y_i' < n$.  Thus if we define
  $X' = \biguplus_{i=0}^n X_i'$ and $U = \biguplus_{i=1}^n U_i'$
  then $X = X' \uplus U$ and $\dim X' < n$.  Similarly, we can define
  $Y' = \biguplus_{i=0}^n Y_i'$ and $V = \biguplus_{i=1}^n V_i'$.
  By definition $U$ and $V$ are piecewise isomorphic and we are done.
\end{proof}

The next theorem shows that the spectral sequence constructs an obstruction to
$\varphi$ extending to a piecewise isomorphism, thereby answering the question
of Gromov in \cite[p121]{gromov99}.  It also directly implies
Theorem~\ref{thm:diff=ker}.

\begin{theorem} \lbl{thm:ll<=>diffs}
  A birational automorphism $\varphi$ of an irreducible variety $X$ extends to a
  piecewise automorphism of $X$ if and only if $d_r[\varphi] = 0$ for all $r\geq
  1$.
\end{theorem}

\begin{proof}
  Suppose that $\dim X = q$.  Since $X$ is irreducible it is represented by a
  class in $B_q$, and thus participates in the spectral sequence.  Note that in
  the rest of the proof all unions are disjoint; however, since we are not
  taking coproducts of varieties but rather thinking of a decomposition of a
  given variety, we use the symbol $\uplus$ rather than $\amalg$.  Suppose that
  $\varphi$ is defined as an isomorphism $U \rto V$ for $U \subseteq X$ and $V
  \subseteq Y$.

  First, suppose that $\varphi$ extends to a piecewise automorphism.  Then
  $[X\bs V] = [X \bs U]$ in $K_0[\V^{(q-1)}]$ and thus $\partial[\varphi] = 0$.
  Then by Lemma~\ref{lem:partial-0}, $d_r[\varphi] = 0$ for all $r \geq 1$.

  Now suppose that $d_r[\varphi] =0$ for all $r \geq 1$.  Since $d_1[\varphi] =
  0$ it follows that $[X \bs U] = [X \bs V]$ in $K_0(\V^{(q,q-1)})$.  We show by
  induction that we can write $X = U_r \uplus X'_r = V_r \uplus Y'_r$ with $U_r$ and
  $V_r$ piecewise isomorphic, $\dim X'_r,\dim Y'_r < n-r$ and $\partial[\varphi]
  = [Y_r'] - [X_r']$.  Then setting $r = n$ gives us the desired result.

  The base case $r=0$ is given to us by the definition of $\varphi$: define $U_0
  = U$ and $V_0 = V$, and $X'_0 = X \bs U_0$, $Y'_0 = X \bs V_0$.  We assume the
  case for $r-1$, and prove it for $r$.  By the inductive hypothesis
  $\partial[\varphi] = [Y_{r-1}']-[X_{r-1}']$ we have
  \[d_r[\varphi] = [Y_{r-1}']-[X_{r-1}']\]  in $K_0(\V^{(n-r,n-r-1)})$. Since
  $d_r[\varphi] = 0$ by assumption, Lemma~\ref{lem:ll<=>diffs} applies and we
  can write $X_{r-1}' = U'_r \uplus X_r'$ and $Y_{r-1}' = V'_r \uplus Y_r'$ with
  $U'_r \cong V'_r$ and $\dim X_r', \dim Y_r' < n-r$.  If we set $U_r = U_{r-1}
  \uplus U_r'$ and $V_r = V_{r-1} \uplus V_r'$ then $U_r$ and $V_r$ are
  piecewise isomorphic.  Thus all that we need to check to finish the induction
  is the formula for $\partial[\varphi]$.  This is straightforward, as
  \[\partial[\varphi] = [Y_{r-1}'] - [X_{r-1}'] = [V_r'] + [Y_r'] - [U_r'] -
  [X_r'] = [Y_r'] - [X_r'],\]
  since $U_r'$ and $V_r'$ are isomorphic.
\end{proof}

\begin{remark} \lbl{rem:z20} In Figure~\ref{fig:ss-kv} each summand in the
  $\pi_1$-column has a $\Z/2$ component, which we ignored in the above
  discussion.  This is because $d_r$ is uniformly zero on all of these.  To see
  this, note that we have morphisms of assemblers
  \[\S \rto \S_G \rto \S\]
  which take $\initial$ to $\initial$ and $*$ to $*$.  The composition of these
  morphisms is the identity.  We know that $K(\S_G) \simeq \Sigma_+^\infty BG
  \simeq \S \vee \Sigma^\infty BG$, with the $\S$ coming from the $K(\S)$ that
  is split off by the above sequence. This $\S$ keeps track of the combinatorial
  details of what is going on: its $\pi_0$ is $\Z$ and corresponds to the number
  of varieties we are considering, and $\pi_1$ keeps track of the permutations
  of the varieties (and corresponds to the sign of the permutation.)  The $\Z/2$
  in each component of $\pi_1$ comes from the $\S$ summand, and using this we
  write down representatives for the $\Z/2$ indexed by $\alpha$: it is
  represented by the two families
  \[\{X \sqcup X \rto^\tau X \sqcup X\} \cup \{\}
  \quad\hbox{and}\quad \{X \sqcup X \rto^= X \sqcup
  X\} \cup \{\}\] where $[X]=\alpha$ in $B_q$. Thus $\partial$ on
  this generator is $0$.
\end{remark}

\section{Multiplication by $\LL$} \lbl{sec:2ndss}

Now consider the morphism of assemblers $L:\V \rto \V$ which sends the variety
$X$ to $X\times \A^1$, and let $C$ be the cofiber of the map $K(L)$ of spectra.
Then we have a long exact sequence in homotopy given by
\begin{equation} \label{eq:LEShtpy}
{K_1(\V) \rto^{p_1} \pi_1C \rto K_0(\V) \rto^{\cdot \LL}
  K_0(\V) \rto \pi_0C \rto 0.}
\end{equation}
The cokernel of $p_1$ is equal to the kernel of multiplication by
$\LL$.  Thus $\LL$ is a zero divisor if and only if $p_1$ is a surjection.

We thus want to analyze the homotopy type of $C$ together with the image of
$p_1$.  It turns out that it is possible to construct $C$ as the $K$-theory of a
simplicial assembler and to analyze this map through the structure of
assemblers.  Applying Theorem~\ref{thm:cofiber}, we can write $C = K(\V/L)$. By
the construction in Theorem~\ref{thm:cofiber}, $\V^{(n)}/L$ is a simplicial
subassembler of $\V^{(n+1)}/L$, and (as mentioned in Remark~\ref{rem:inc-cofib})
the inclusion $\V^{(n)}/L \rcofib \V^{(n+1)}/L$ gives a cofibration on
$K$-theory.  We get the following commutative diagram:
\begin{equation} \label{diag:filtspec}
\begin{tikzpicture}[baseline]%
\matrix (m) [matrix of math nodes, row sep=2.5em,%
column sep=2.5em, text height=2.2ex, text depth=0.7ex]
  { \cdots & K(\V^{(n-2)}) & K(\V^{(n-1)}) & \cdots & K(\V) \\
    \cdots & K(\V^{(n-1)}) & K(\V^{(n)}) & \cdots & K(\V) \\
    \cdots & K(\V^{(n-1)}/L) & K(\V^{(n)}/L) & \cdots & K(\V/L) \\};
  \cofib{1-1}{1-2} \cofib{1-2}{1-3} \cofib{1-3}{1-4} \cofib{1-4}{1-5} 
  \cofib{2-1}{2-2} \cofib{2-2}{2-3} \cofib{2-3}{2-4} \cofib{2-4}{2-5} 
  \cofib{3-1}{3-2} \cofib{3-2}{3-3} \cofib{3-3}{3-4} \cofib{3-4}{3-5} 
  \to{1-2}{2-2}^{K(L)} \to{1-3}{2-3}^{K(L)} \to{1-5}{2-5}^{K(L)}
  \to{2-2}{3-2} \to{2-3}{3-3} \to{2-5}{3-5}
\end{tikzpicture}
\end{equation}

We define the function
\[\ell:B_{n-1} \rto B_{n}\qquad\hbox{by}\qquad \ell[X] = [X\times\A^1].\]

\begin{proposition} \lbl{prop:Cbetacofib}
  For $\beta$ in $B_n$, let $\nabla_\beta: \bigvee_{\ell^{-1}(\beta)} \S \rto \S$
  be the fold map, and let 
  \[C_\beta = \hocofib \nabla_\beta \vee \hocofib \left( \bigvee_{\alpha\in
      \ell^{-1}(\beta)} \Sigma^\infty B\Aut(\alpha) \rto \Sigma^\infty
    B\Aut(\beta)\right).\] There exists a spectral sequence
  \begin{equation}
    \label{eq:kv/lss}
    \tilde E^1_{p,q} = \bigoplus_{\beta\in B_q} \pi_p C_\beta \Rto \pi_p K(\V/L).
  \end{equation}
\end{proposition}

\begin{proof}
  It suffices to show that
  \[\makeshort{\hocofib\Big(K(\V^{(n-1)}/L) \rto K(\V^{(n)}/L)\Big) \simeq \bigvee_{\beta\in B_n} C_\beta.}\]
  Then the spectral sequence from (\ref{eq:genss}) for the filtration on
  $K(\V/L)$ is the one in the statement of the proposition.  Therefore
  we focus on proving the claim.  For this proof, we restrict our attention to
  the subassembler of $\V$ consisting of only irreducible varieties; in order to
  avoid cluttering the notation, we still denote it by $\V$.  By
  Theorem~\ref{thm:1stapprox} this gives the correct homotopy type.
  
  Write $\iota: \V^{(n-1)} \rcofib \V^{(n)}$, and $\tilde \iota: V^{(n-1)}/L
  \rto \V^{(n)}/L$.  Then we have a diagram
  \begin{equation} \label{diag:total-cofiber}
    \begin{tikzpicture}[baseline]%
      \matrix (m) [matrix of math nodes, row sep=2.5em,%
      column sep=2.5em, text height=2.2ex, text depth=0.7ex] 
      { \V^{(n-2)} & \V^{(n-1)} & V^{(n-1)}/\iota \\
        \V^{(n-1)} & \V^{(n)} & \V^{(n)}/\iota \\
        \V^{(n-1)}/L & \V^{(n)}/L & (\V^{(n)}/L)/\tilde\iota \\};
      \cofib{1-1}{1-2}^{\iota} \cofib{2-1}{2-2}^{\iota} \to{1-1}{2-1}^{L}
      \to{1-2}{2-2}^{L} \to{1-3}{2-3}^{\tilde L} \to{1-2}{1-3} \to{2-2}{2-3} \to{2-1}{3-1}
      \to{2-2}{3-2} \cofib{3-1}{3-2}^{\tilde \iota} \to{3-2}{3-3} \to{2-3}{3-3}
    \end{tikzpicture}
  \end{equation}
  From Definition~\ref{def:simp-cofib} we have
  $(\V^{(n)}/L)/\tilde\iota \cong (\V^{(n)}/\iota)/\tilde L$.
  Therefore if we want to determine the homotopy type of
  $K((\V^{(n)}/L)/\tilde\iota)$ it suffices to determine the cofiber of
  the map
  \[K(\tilde L): K(\V^{(n-1)}/\iota) \rto K(\V^{(n)}/\iota).\] We have a morphism of
  assemblers $V^{(n)}/\iota \rto \V^{(n,n-1)}$ given by mapping a
  variety of dimension $n$ to itself and a variety of dimension less than $n$ to
  the initial object.  Then the diagram
  \begin{diagram-fixed}
    { \V^{(n-1)}/\iota & \V^{(n-1,n-2)} \\
      \V^{(n)}/\iota & \V^{(n,n-1)} \\ };
    \to{1-1}{2-1}_{\tilde L} \to{1-2}{2-2}^{\tilde L}
    \we{1-1}{1-2} \we{2-1}{2-2}
  \end{diagram-fixed}
  commutes, and by Theorem~\ref{thm:cofib} the two horizontal morphisms are weak
  equivalences after applying $K$-theory.  Thus after applying $K$-theory, the
  cofiber of the left-hand map $\tilde L$ is the same as the 
  cofiber of the right-hand map $\tilde L$.

  For $\beta$ in $B_n$, let $(\V^{(n,n-1)})|_\beta$ be the full subassembler of
  those varieties whose birational isomorphism class is $\beta$.  Then we can
  rewrite the right-hand column as
  \[\bigvee_{\beta\in B_n} \left(\bigg( \bigvee_{\alpha \in \ell^{-1}(\beta)}
    \V^{(n-1,n-2)}|_\alpha\bigg) \rto \V^{(n,n-1)}|_\beta\right).\] Since
  $K$-theory and cofibers commute with coproducts, it suffices to show
  that the cofiber of each term in the above wedge product is weakly
  equivalent to $C_\beta$.  We want to show that the splitting from
  Theorem~\ref{thm:graded} is compatible with this map, so that we can compute
  cofibers in terms of homomorphisms of groups instead of in terms of
  maps of spaces.
  
  We therefore aim to construct a splitting of both sides of the map
  simultaneously.  Fix a representative $X_\beta$.  For each $\alpha$ in
  $\ell^{-1}(\beta)$, choose an isomorphism $\psi_\alpha: \Aut_k k(X_\alpha\times
  \A^1)\rto \Aut_k k(X_\beta)$. For any variety $Y$ with a chosen embedding $Y
  \rcofib \A^N$ we let $\C_Y$ be the assembler of subvarieties of $Y$,
  considered as algebraic subsets of $Y(\bar k)$ defined over $k$.  We then have
  the following diagram:
  \begin{diagram-fixed}
    { \S_{\Aut_kk(X_\alpha)} & \S_{\Aut_kk(X_\alpha\times \A^1)} &  \S_{\Aut_kk(X_\beta)} \\
      \C_{X_\alpha} & \C_{X_\alpha \times \A^1}  \\
      \V^{(n-1,n-2)}|_\alpha & \V^{(n,n-1)}|_\beta  \\}; 
    \to{1-1}{1-2}^\varphi \to{2-1}{2-2} \to{3-1}{3-2} 
    \we{2-1}{1-1} \we{2-2}{1-2} \we{2-1}{3-1} \we{2-2}{3-2} 
    \to{1-2}{1-3}^{\psi_\alpha}
  \end{diagram-fixed}
  Note that the morphism across the bottom is a component of the map we're
  trying to find the cofiber of.  The morphisms in the top row are all induced
  by group homomorphisms. The first is induced by the inclusion of
  $\Aut_kk(X_\alpha)$ into $\Aut_kk(X_\alpha\times\A^1)$.  By \cite[Theorem
  2.1(2)]{Z-Kth-ass}, the top square commutes.  The vertical maps from the
  second row to the third row take each subvariety to itself; they are
  equivalences after applying $K$ by Theorem~\ref{thm:1stapprox}.  The bottom
  square commutes by definition.

  From this diagram we can conclude that taking the cofiber of
  the bottom row is weakly equivalent to taking the cofiber of the top
  row, and we see that
  \[ \hocofib \bigvee_{\alpha\in \ell^{-1}(\beta)} K(\V^{(n-1,n-2)}|\alpha) \rto
  K(\V^{(n,n-1)}|_\beta) \simeq \hocofib\bigvee_{\alpha \in \ell^{-1}(\beta)}
  K(\S_{\Aut(\alpha)}) \rto K(\S_{\Aut(\beta)}).\]

  We now compute this cofiber.  We would like the cofiber to
  split further into a part that just comes from a homomorphism of groups, and a
  part that comes from a fold map of sphere spectra.  Note that we have the
  following commutative diagram of assemblers:
  \begin{diagram}
    { \bigvee_{\alpha\in \ell^{-1}(\beta)} \S & \bigvee_{\alpha \in
        \ell^{-1}(\beta)} \S_{\Aut(\alpha)} &
      \bigvee_{\alpha\in \ell^{-1}(\beta)} \S \\
      \S & \S_{\Aut(\beta)} & \S\\ };
    \to{1-1.mid east}{1-2.mid west} \to{1-2.mid east}{1-3.mid west} 
    \to{2-1.mid east}{2-2.mid west} \to{2-2.mid east}{2-3.mid west}
    \to{1-1}{2-1}_{f_\beta} 
    \to{1-2}{2-2}^L \to{1-3}{2-3}^{f_\beta}
  \end{diagram}
  The maps $f_\beta$ are the fold maps, and $K(f_\beta) \simeq \nabla_\beta$.  The
  horizontal morphisms are induced by the group homomorphisms $1 \rto G \rto 1$.
  Note that the compositions along the top and bottom rows are identity
  morphisms.  Thus the cofiber of $K(L)$ splits as the wedge of the
  cofiber of $\nabla_\beta$ with the cofiber of
  \[\bigvee_{\alpha\in \ell^{-1}(\beta)} (\makeshort{\hocofib K(\S) \rto
    K(\S_{\Aut(\alpha)})}) \rto (\makeshort{\hocofib K(\S) \rto K(\S_{\Aut(\beta)})}).\]
  The desired statement follows from the observation that
  $\hocofib (K(\S) \rto K(\S_G)) \simeq \Sigma^\infty BG$.
\end{proof}

This is a much more complicated splitting result than the one we got in
Theorem~\ref{thm:graded}, but it turns out that we can still write down its
$\pi_0$ and $\pi_1$.  For simplicity, we write
\[\tilde C_\beta \defeq  \hocofib\bigvee_{\alpha\in
  \ell^{-1}(\beta)} \Sigma^\infty B\Aut(\alpha) \rto \Sigma^\infty
B\Aut(\beta),\]
so that $C_\beta = (\hocofib \nabla_\beta) \vee \tilde C_\beta$.  

First we compute $\pi_i(\hocofib\nabla_\beta)$ for $i=0,1$.  If
$\ell^{-1}(\alpha) = \eset$, then $\hocofib \nabla_\beta = \S$, so $\pi_0 = \Z$
and $\pi_1 = \Z/2$.  On the other hand, if $\ell^{-1}(\alpha) \neq 0$, then
$\pi_0(\nabla_\beta): \Z^{\oplus \ell^{-1}(\beta)} \rto \Z$ is addition, and in
particular is surjective; therefore $\pi_0 = 0$.  Thus we see that
\[\pi_0 \hocofib \nabla_\beta \cong
\begin{cases}
  \Z & \hbox{if }\ell^{-1}(\beta) = \eset \\
  0 & \hbox{otherwise.}
\end{cases}
.\] Now consider $\pi_1$.  The map $\nabla_\beta$ has a section, so in fact
$\hocofib \nabla_\beta \simeq \bigvee_{|\ell^{-1}(\beta)|-1} \Sigma\S$.  Of
course, $\ell^{-1}(\beta)$ may be infinite, so the indexing on the sum doesn't
quite make sense.  However, we can still clearly identify what $\pi_1 \hocofib
\nabla_\beta$ is: it is the subgroup of $\Z^{\oplus \ell^{-1}(\beta)}$
consisting of those formal sums
$a_1\alpha_1 +\cdots + a_n\alpha_n$
such that $a_1+\cdots+a_n = 0$.  Since
this is a subgroup of ``codimension'' $1$ the indexing is justified from that
perspective.  Thus in general, $\pi_1(\hocofib \nabla_\beta) \cong \tilde
\Z^{\oplus \ell^{-1}(\beta)}$, justifying our use of this notation.

To compute $\pi_0$ and $\pi_1$ of $\tilde C_\beta$ we use the long exact
sequence in homotopy.  Then we have a long exact sequence
\[\bigoplus_{\alpha\in \ell^{-1}(\beta)} \Aut(\alpha)^\ab \rto^\ell \Aut(\beta)^\ab \rto
\pi_1\tilde C_\beta \rto 0 \rto 0 \rto \tilde C_\beta \rto 0.\]
Thus we conclude that $\pi_0\tilde C_\beta = 0$ and 
\[\pi_1\tilde C_\beta = \Aut(\beta)^\ab \left/ \iota\bigg(\bigoplus_{\alpha\in
    \ell^{-1}(\beta)} \Aut(\alpha)^\ab\bigg)\right.\] Note that even though
$\Aut(\alpha) \rto \Aut(\beta)$ is only well-defined up to conjugation this
formula is independent of those choices.  Notice, also, that $\pi_1\tilde
C_\beta$ is a quotient of $\Aut(\beta)$, so any element in $\pi_1\tilde C_\beta$
can be represented by a birational automorphism of $X_\beta$.

Using Proposition~\ref{prop:Cbetacofib} we can write down the first two columns
of $\tilde E^1_{*,*}$, which appear in Figure~\ref{fig:ss-kv/l}.
\begin{figure}[h]
  \centering
  \begin{tikzpicture}[xscale=8]
    \node (OO) at (0,0) {$\Z^{\oplus B_0}$};
    \node (Odots) at (0,1) {$\vdots$};
    \node (On-1) at (0,2) {$\Z^{\oplus B_{n-1} \bs \ell(B_{n-1})}$};
    \node (On) at (0,3) {$\Z^{\oplus B_n\bs \ell(B_{n-1})}$};
    \node (FO) at (1,0) {$\bigoplus_{\beta\in B_0} \Aut(\beta)^\ab \oplus \Z/2$};
    \node (Fdots) at (1,1) {$\vdots$};
    \node (Fn-1) at (1,2) {$\bigoplus_{\beta\in
        B_{n-1}}  \pi_1\tilde C_\beta \oplus  \tilde\Z^{\oplus \ell^{-1}(\beta)}$};
    \node (Fn) at (1,3) {$\bigoplus_{\beta\in B_n} \pi_1\tilde C_\beta \oplus
      \tilde\Z^{\oplus \ell^{-1}(\beta)} $};
    
    \node (pi0) at (0,-1) {$\pi_0$};
    \node (pi1) at (1,-1) {$\pi_1$};
    \draw (OO.south west) +(-.1,-.1) to[draw] +(1.5,0);
    \draw[->] (Fn.west) -- node[auto,swap] {$d_1$} (On-1.east);
    \draw[->, densely dotted] (Fn.west) -- node[below right] {$d_{n}$} (OO.east);
  \end{tikzpicture}

  \caption{Spectral sequence for $K((\V/L).)$}
  \label{fig:ss-kv/l}
\end{figure}

\begin{lemma} \lbl{lem:partial-comp}
  Let $\beta$ be in $B_n$.  
  \begin{itemize}
  \item[(1)] For every birational automorphism $\varphi \in \Aut(\beta)$, the
    homomorphism 
    \[\partial: \pi_1\tilde C_\beta \leq \tilde E^1_{1,n} \rto \tilde
    A_{0,n-1}\]
    is given by
    \[\partial[\varphi] = [X_\beta \bs V] - [X_\beta\bs U]\]
    if $\varphi$ is represented as an isomorphism $U \rto V$ of open subsets of
    $X_\beta$.

  \item[(2)] If $\ell^{-1}(\beta) = \eset$ then $\partial(\Z/2) = 0$.  

  \item[(3)] Suppose that $\ell^{-1}(\beta) \neq \eset$.  The
    homomorphism
    \[\partial: \pi_1\hocofib \nabla_\beta \leq \tilde E^1_{1,n} \rto \tilde A_{0,n-1}\]
    is given by
    \[\partial([\alpha]-[\alpha']) = [(X_{\alpha'}\times \A^1)\bs V] -
    [(X_{\alpha} \times \A^1)\bs U]\] for any formal difference
    $[\alpha]-[\alpha']$ and for any birational isomorphism $\varphi:
    X_\alpha\times \A^1 \rdash X_{\alpha'} \times \A^1$ represented by an
    isomorphism $U \rto V$.
  \end{itemize}
\end{lemma}

\begin{proof}
  \noindent
  Proof of (1): The group $\tilde E^1_{1,n}$ is $K_1((\V^{(n)}/L)/\tilde\iota)$.
    Since $(V^{(n)}/L)/\tilde\iota \cong (V^{(n)}/\iota)/\tilde L$, to compute
    representatives for its elements it suffices to compute representatives for
    the elements of $K_1((\V^{(n)}/\iota)/\tilde L)$ and check that they
    represent the correct classes.  Since we obtained our calculation of these
    groups from the long exact sequence for the cofiber sequence
    \[K(\V^{(n-1)}/\iota) \rto K(\V^{(n)}/\iota) \rto K((\V^{(n)}/\iota)/\tilde
    L)\] it suffices to check that they are correct by checking them in this
    sequence; thus to represent $[\varphi]$ it suffices to construct a
    representative of $[\varphi]$ in $K_1(\V^{(n)}/\iota)$ and then compute its
    image in $K_1((\V^{(n)}/\iota)/\tilde L)$.

    By using the model of the cofiber constructed in Theorem~\ref{thm:cofiber}
    together with \cite[Corollary 3.10]{Z-ass-pi1} we know that the following data
    represents an element of $E^1_{1,n}$.  Let $\beta$ be in $B_n$ and let $\varphi$
    be a birational automorphism of $X_\beta$, represented by an isomorphism $U
    \rto V$ of open subsets of $X_\beta$.  Then a representative of the class of
    $\varphi$ in $\pi_1\tilde C_\beta$ is given by the covering families
    \[\makeshort{\{U \rto X_\beta\} \cup \{X_\beta \bs U \rto X_\beta\}\quad\hbox{and}
    \quad \{U \rto^{\cong} V \rto X_\beta\} \cup \{X_\beta \bs V \rto
    X_\beta\}}.\]
    Thus
    $\partial[\varphi] = [X_\beta \bs V] - [X_\beta \bs U]$.

    \noindent 
    Proof of (2):  As in part (1), we can construct a representative for the nonzero
    element of $\Z/2$ by finding one in $K_1(\V^{(n)}/\iota)$ and then
    considering its image.  Therefore the nonzero element of $\Z/2$ is
    represented by two copies of $X_\beta$ being swapped, and thus has two
    complete covering families.  Therefore $\partial$ on it is zero.

    \noindent 
    Proof of (3): To find a representative of $[\alpha]-[\alpha']$ it suffices to
    construct an element which maps to $[\alpha]-[\alpha']$ in
    $K_0(\V^{(n-1)}/\iota)$ under $\partial$.  Thus we can represent it
    $[\alpha]-[\alpha']$ by the following data: varieties $X_\alpha$ and
    $X_{\alpha'}$, and a birational isomorphism $\varphi: X_\alpha\times\A^1
    \rdash X_{\alpha'} \times \A^1$ given as an isomorphism $U \rto V$.  Then by
    \cite[Remark 3.9]{Z-ass-pi1} the corresponding element of $\pi_1$ is
    represented by the covering families
    \[\makeshort{\{U \rto (X_\alpha\times \A^1)\sqcup Z_{\alpha'}, Z_\alpha \rto
      (X_\alpha\times \A^1)\sqcup
      Z_{\alpha'}, Z_{\alpha'} \rto (X_\alpha\times \A^1) \sqcup Z_{\alpha'}\}}\]
    and
    \[\makeshort{\{U \rto^{\cong} V \rto (X_{\alpha'}\times \A^1)\sqcup Z_{\alpha}, Z_\alpha \rto
      (X_{\alpha'}\times \A^1)\sqcup
      Z_{\alpha}, Z_{\alpha} \rto (X_{\alpha'}\times \A^1) \sqcup Z_{\alpha}\},}\]
    where $Z_\alpha = (X_\alpha\times\A^1) \bs U$ and $Z_{\alpha'} =
    (X_{\alpha'}\times \A^1) \bs V$.\footnote{Rephrasing this in the notation
      of \cite[Remark 3.9]{Z-ass-pi1}, we use $B_0 = 0$, $C_0 = \{U, Z_\alpha,
      Z_{\alpha'}\}$, $V_1 = (\{X_\alpha\times \A^1,Z_{\alpha'}\},
      \{X_{\alpha'}\}, \{Z_\alpha\},0)$ and $W_1 = (\{X_{\alpha'} \times \A^1,
      Z_\alpha\}, \{X_\alpha\}, \{Z_{\alpha'}\},0)$.}  Then
    $\partial([\alpha]-[\alpha']) = [Z_{\alpha'}]-[Z_\alpha]$.
\end{proof}

\begin{proposition} \lbl{prop:ssmap} The map $K(\V) \rto K(\V/L)$ induces a map
  of spectral sequences $E^*_{p,q} \rto \tilde E^*_{p,q}$.  This map is
  surjective on $\tilde E^1_{0,n}$ and surjective onto the component
  $\bigoplus_{\beta\in B_n} \pi_1\tilde C_\beta$ in $\tilde E^1_{1,n}$.
\end{proposition}

\begin{proof}
  The first part follows directly from the fact that this map gives rise to a
  map of filtered spectra.  The statement about $\tilde E^1_{0,n}$ is true for
  any such spectral sequence arising from a connective spectrum filtered by
  connective spectra, since it follows from the fact that for any map of
  connective spectra $f:X \rto Y$, $\pi_0 Y \rto \pi_0 \hocofib f$ is
  surjective.  The last statement follows by checking that in the map $K_1(\V)
  \rto K_1(\V/L)$ the class of a birational automorphism $\varphi$ maps to
  itself, as mentioned in the proof of Lemma~\ref{lem:partial-comp}.
\end{proof}

\section{Restricting to convenient fields} \lbl{sec:convenient}

All of the results mentioned so far in the paper are independent of the field
$k$ and have thus contained rather little algebraic geometry.  We now want to
compute some of the differentials in the spectral sequence for $K(\V/L)$.  For
this we need to restrict our attention to those fields where the
computations are feasible.

\begin{definition} \lbl{def:convenient}
  A birational isomorphism $\varphi:
  X \rdash Y$ between smooth projective varieties $X$ and $Y$ of dimension $n$
  over $k$ is \textsl{convenient} if 
  \[[X \bs U] - [Y \bs V] \in \mathrm{im}\Big(K_0(\V^{(n-2)}) \rto^{\cdot \LL}
  K_0(\V^{(n-1)})\Big).\] Here $U\subseteq X$ and $V\subseteq Y$ are open and
  $\varphi$ is an isomorphism $U \rto V$.  

  A field $k$ is \textsl{convenient} if all birational isomorphisms between
  smooth projective varieties over $k$ are convenient.
\end{definition}

A priori, it is not obvious that the definition of a convenient birational
isomorphism is only dependent on the birational isomorphism and not on the
choices of $U$ and $V$.  In order to check that this is not the case it suffices
to check that if $\varphi$ is convenient when represented as an isomorphism $U
\rto V$ then it is also convenient when represented as an isomorphism $U' \rto
V'$ for any open $U' \subseteq U$.  To see this, observe that 
\begin{eqnarray*}
  [X \bs U'] - [Y \bs V'] &=& [X \bs U] + [U \bs U'] - [Y \bs V]  - [V \bs V']
  \\
  &=& \underbrace{[X \bs U] - [Y \bs V]}_{\in \mathrm{im}(\cdot \LL)} +
  \underbrace{[U \bs U'] - [V \bs V']}_{=0}.
\end{eqnarray*}
The equality $[U \bs U'] = [V \bs V']$ is because $\varphi$ restricts to an
isomorphism $U \bs U' \rto V \bs V'$.

All computations in this section are going to be done in convenient fields.
Currently, the only fields known to be convenient are fields with characteristic
$0$.  However, as the proof of this (see below) relies only on the Weak
Factorization Theorem, any field for which the Weak Factorization Theorem is
known is convenient.  We introduce this definition to highlight that this is the
only property needed for all computations in this section.

\begin{lemma} 
  Fields with characteristic $0$ are convenient.
\end{lemma}

\begin{proof}
  If $k$ has characteristic $0$ then the Weak Factorization Theorem holds in $k$
  \cite[Section 5.20]{wlodarczyk09}.  Thus any birational isomorphism $\varphi:
  X \rdash Y$ can be factored as a composition of blow-ups and blow-downs along
  smooth centers.  We prove that all birational isomorphisms over $k$ are
  convenient by induction on the number of blowups and blowdowns in this
  factorization.

  We first prove the base case.  Suppose that $Y$ is a blowup of $X$ along $Z$.
  Then in $K_0(\V_k^{(n)})$ we can write
  $[X] = [U] + [Z]$ and $[Y] = [U] + [Z']$, where $Z'$ is a $\P^\ell$ bundle on $Z$
  for some $1 \leq \ell \leq n-1-\dim Z$.  However, in $K_0[\V^{(\dim
    Z+\ell)}]$ and thus in $K_0[\V^{(n-1)}]$, $[Z'] = [Z\times \P^\ell]$.  Thus
  \begin{align*}
    [X \bs U] - [Y \bs U] &= [Z] - [Z\times \P^\ell] = [Z] - ([Z] + [Z]\LL +
  \cdots + [Z]\LL^\ell) \\
  &= -\LL([Z] + \cdots + [Z]\LL^{\ell-1}) \in \LL K_0(\V^{(n-2)}).
  \end{align*}
  For a blowdown the computation is the same with a reversed sign.  

  Now we prove the induction step.  Suppose that $\varphi = \varphi''\varphi'$,
  where $\varphi'$ and $\varphi''$ require fewer blowups/blowdowns than
  $\varphi$.  Represent $\varphi':X \rdash Y$ by $U \rto^{\cong} U'$ and
  $\varphi'':Y \rdash Z$ by $V \rto^{\cong} V'$.  By the induction hypothesis
  both $\varphi'$ and $\varphi''$ are convenient.

  Suppose first that $U' = V$.  Then $\varphi$ can be represented as an
  isomorphism $U \rto V'$, and we have
  \[[Z \bs V'] - [X \bs U] = ([Z \bs V'] - [Y \bs V]) + ([Y \bs U'] - [X \bs U])
  \in \mathrm{im}(\cdot \LL).\]
  Thus in this case $\varphi$ is convenient.

  Now consider the general case.  Let $W = \varphi'(U) \cap V \subseteq Y$.
  Then $\varphi'$ can be represented as an isomorphism $(\varphi')^{-1}(W) \rto
  W$ and $\varphi''$ can be represented as an isomorphism $W \rto
  \varphi''(W)$.  This reduces to the above case, and we see that $\varphi$ must
  also be convenient.
\end{proof}

\begin{definition}
  We define the \textsl{line degree} of a birational isomorphism class
  $\alpha$ in $B_n$, denoted $\deg \alpha$ to be the maximum integer $s$ such that
  $\alpha$ in $\ell^s(B_{n-s})$.  We define the equivalence relation $\sim_r$ on
  $B_n$ by defining $\alpha \sim_r \alpha'$ if $\ell^r(\alpha) =
  \ell^r(\alpha')$; we write $\alpha\sim_\infty \alpha'$ if there exists $r$
  such that $\alpha \sim_r \alpha'$.
\end{definition}

By definition, $\alpha$ and $\alpha'$ in $B_n$ are stably birational if and only if
$\alpha\sim_\infty \alpha'$.

Our first result computes the differentials in the spectral sequence for
$K(\V/L)$
in a more convenient form than that given by
Lemma~\ref{lem:partial-comp}.

\begin{lemma} \lbl{lem:nice-dcomp} Let $k$ be a convenient field, and let
  $\beta$ be in $B_n$.  Consider the spectral sequence $\tilde E^*_{p,q}$.
  \begin{itemize}
  \item[(1)]   For every
    birational automorphism $\varphi$ in $\Aut(\beta)$ and all $r \geq 1$,
    $d_r[\varphi] = 0$.
  \item[(2)] If $\ell^{-1}(\beta) = \eset$ then $d_r(\Z/2) = 0$ for all $r \geq 1$.
  \item[(3)] Suppose $\ell^{-1}(\beta) \neq \eset$.  For distinct
    $\alpha,\alpha'$ in $\ell^{-1}(\beta)$,
    \[d_r([\alpha]-[\alpha']) =
    \begin{cases}
      0 & 1\leq r \leq \min(\deg \alpha, \deg \alpha') \\
      [(\ell^{r-1})^{-1}(\alpha)] - [(\ell^{r-1})^{-1}(\alpha')] & r  =
      \min(\deg \alpha, \deg \alpha') + 1.
    \end{cases}\]
    Here, $(\ell^{r-1})^{-1}(\alpha)$ is any preimage of $\alpha$ under $\ell^{r-1}$.
  \item[(4)] For all $r\geq 1$,
    \[\tilde E^r_{0,n} \cong \Z^{\oplus (B_n/\sim_{r-1}) \bs \ell(B_{n-1})}.\]
  \item[(5)] For $\alpha,\alpha'$ as in part (3),
    $d_{\min(\deg \alpha, \deg\alpha')+1}([\alpha]-[\alpha']) \neq 0$.
  \end{itemize}
\end{lemma}

In particular, the formula in (3) and the statement in (5) tell us that the
length of a differential killing a difference $[\alpha]-[\alpha']$ gives us the
``minimal stability degree'' of $\alpha$ and $\alpha'$: if $d_r$ hits
$[\alpha]-[\alpha']$ then the birational isomorphism classes $\alpha$ and
$\alpha'$ are birational after multiplication with $\P^r$ but not after
multiplication with $\P^{r-1}$.

\begin{proof}
  \noindent
  Proof of (1): The differentials in the spectral sequence are obtained from the
  boundary map in the long exact sequence.  Thus to show that $d_r[\varphi] = 0$
  for all $r$ it suffices to show that $\partial [\varphi]$ is $0$, where
  $\tilde\iota: \V^{(n-1)}/L \rto \V^{(n)}/L$ and $\partial$ is the connecting
  homomorphism
  \[ K_1((\V^{(n)}/L)/\tilde\iota) \rto^\partial K_0(\V^{(n-1)}/L)\] in the long
  exact sequence in homotopy of the cofiber sequence
  \[K(\V^{(n-1)}/L) \rto K(\V^{(n)}/L) \rto K((\V^{(n)}/L)\tilde\iota).\] From
  Lemma~\ref{lem:partial-comp} and the fact that $k$ is convenient we know that
  \[\partial[\varphi] = [X_\beta\bs V] - [X_\beta \bs U]\]
  is in $\LL K_0(\V^{(n-2)})$.
  However, 
  \[K_0(\V^{(n-1)}/L) \cong K_0(\V^{(n-1)}) / \LL K_0(\V^{(n-2)}),\]
  so we see that $\partial[\varphi] =0$.

  \noindent
  Proof of (2): By Lemma~\ref{lem:partial-0} it suffices to show that
  $\partial(\Z/2) = 0$.  This is proved in Lemma~\ref{lem:partial-comp}(2).

  \noindent
  Proof of (3): From lemma~\ref{lem:partial-comp}, we know that 
    \[\partial([\alpha]-[\alpha']) = [(X_{\alpha'}\times\A^1)\bs V] -
    [(X_\alpha\times \A^1) \bs U]\] for any birational isomorphism $\varphi:
    X_\alpha\times \A^1 \rdash X_{\alpha'} \times \A^1$.  Note that we can think
    of $\varphi$ as a birational isomorphism $X_\alpha\times\P^1 \rdash
    X_{\alpha'} \times\P^1$; then since $k$ is convenient we know that 
    \[[(X_{\alpha'}\times \P^1) \bs V] - [(X_{\alpha}\times \P^1)\bs U]\] is in $\LL
    K_0(\V^{(n-2)})$.  Since 
    $(X_{\alpha'} \times \P^1) \bs V = ((X_{\alpha'}\times \A^1) \bs V) \uplus
    X_{\alpha'}$
    we conclude that in $K_0(\V^{(n-1)})$
    \begin{eqnarray*}
      && [(X_{\alpha'}\times \A^1) \bs V] - [(X_\alpha\times \A^1)\bs U] 
      \equiv [X_\alpha] - [X_{\alpha'}] \pmod{\LL K_0(\V^{(n-2)})}.
    \end{eqnarray*}
    Thus we have shown that 
    $\partial([\alpha]-[\alpha']) = [X_\alpha]-[X_{\alpha'}]$.
    Now suppose that there exists an $s$ such that $\alpha = \ell^s(\gamma)$ and
    $\alpha' = \ell^s(\gamma')$.  Since $X_\alpha$, $X_{\alpha'}$ are
    arbitrary smooth projective representatives of the birational isomorphism
    classes we could have chosen
    $X_\alpha = X_\gamma \times \P^s$ and $X_{\alpha'} = X_{\gamma'} \times
    \P^s$.
    Then 
    $[X_{\alpha}] - [X_{\alpha'}] \equiv [X_\gamma] - [X_{\gamma'}] \pmod{\LL
      K_0(\V^{(n-2)})}$. 

  \noindent
  Proof of (4,5): In order to prove (5) it suffices to prove the following
  statement for all $r$:
    \begin{itemize}
    \item[(5a)] If $\alpha \neq \alpha'$ in $\ell^{-1}(\beta)$ and
      $\min(\deg \alpha, \deg \alpha') < r$ then
      $d_{\min(\deg\alpha, \deg \alpha')+1}([\alpha]-[\alpha']) \neq 0$.
    \end{itemize}
    We prove (4) and (5a) by a paired induction on $r$.  For $r=1$, statement
    (4) is just the computation of $\tilde E^1_{0,n}$.  Statement (5a) just says
    that if one of $\alpha,\alpha'$ is not in $\ell(B_{n-1})$ then
    $[\alpha]-[\alpha']\neq 0$, which follows from the formula for $\tilde
    E^1_{0,n}$. Now suppose that both statements are true up to $r$ and consider
    $d_{r}$.  By the induction hypothesis we know that for all $m$
    \[\tilde E^{r}_{0,m} \cong \Z^{\oplus (B_m /\sim_{r-1}) \bs \ell(B_{m-1})}.\]
    Note that the free part of $\tilde E^{r}_{1,n}$ is generated by all
    $\alpha$ in $\ell^{-1}(\beta)$ such that $\deg \alpha \geq r-1$, since by (5a)
    all differences where at least one term had degree less than $r-1$ had a
    nonzero differential on it.  Thus to prove (4) we need to show that the
    cokernel of
    \[d_{r}: \bigoplus_{\substack{\beta\in B_n \\ \ell^{-1}(\beta) \neq \eset}}
    \tilde\Z^{\oplus \{\alpha\in \ell^{-1}(\beta) \,|\, \deg \alpha \geq r-1\}}
    \rto \Z^{\oplus (B_{n-r}/\sim_{r-1}) \bs \ell(B_{n-r-1})}\] is $\Z^{\oplus
      (B_{n-r}/\sim_{r}) \bs \ell(B_{n-r-1})}$.  (Note that restricting to
    those $\beta$ with nonzero preimage doesn't change the cokernel; we do this
    purely for convenience.)  For any $\beta$, suppose that
    $\alpha,\alpha'\in\ell^{-1}(\beta)$ are distinct.  If $\deg \alpha, \deg
    \alpha' > r-1$ then by (3) $d_r([\alpha]-[\alpha']) = 0$, so they do not
    influence the cokernel of $d_r$.  Thus it suffices to consider those pairs
    $\alpha,\alpha'$ with $\deg \alpha = r-1$.  

    Let $\gamma $ be in $(\ell^{r-1})^{-1}(\alpha)$ and $\gamma'$ be in
    $(\ell^{r-1})^{-1}(\alpha')$.  Note that by definition $\gamma\sim_r \gamma'$
    but $\gamma \not\sim_{r-1} \gamma'$.  On the other hand, if $\gamma\sim_r
    \gamma'$ but $\gamma \not\sim_{r-1}\gamma'$ then
    \[d_r([\ell^{r-1}(\gamma)] - [\ell^{r-1}(\gamma')]) = [\gamma]-[\gamma'],\]
    so any such pair is in the image of $d_r$.  Thus the image of $d_r$
    quotients out by $[\gamma]-[\gamma']$ for those pairs $\gamma$ such that
    $\gamma \not\sim_{r-1} \gamma'$ but $\gamma \sim_r \gamma'$, and we see that
    \[\tilde E^{r+1}_{0,n-r} = \Z^{\oplus(B_{n-r}/\sim_r) \bs \ell(B_{n-r-1})}.\]

    It remains to check (5a): if $\deg \alpha = r-1$ and $\deg \alpha' \geq r-1$
    then $d_r([\alpha]-[\alpha']) \neq 0$.  Suppose that the opposite held, so
    that $[\gamma] - [\gamma'] = 0$.  Then from (4) we know that $\gamma
    \sim_{r-1} \gamma'$, so that $\alpha = \ell^{r-1}(\gamma)  =
    \ell^{r-1}(\gamma') = \alpha'$, a contradiction.  Statement (5a)
    for $r$ follows.
\end{proof}

Lemma~\ref{lem:nice-dcomp} directly implies the following:

\begin{corollary} \lbl{cor:1-surj}
  The homomorphism $E^1_{1,n} \rto \tilde E^1_{1,n}$ is surjective on permanent cycles.
\end{corollary}

As a quick application of Lemma~\ref{lem:nice-dcomp} we prove an extension of
Liu and Sebag's result \cite[Corollary 5(1)]{liusebag}.

\begin{corollary} \lbl{cor:ls} A variety $X$ has \textsl{stable dimensional
    complexity} $k$ if $k$ is the minimal integer such that there exists a
  variety $X'$ of dimension $k$ such that $X$ is stably birational to $X'\times
  \A^{\dim X - k}$.  If X and Y both have dimension $n$ and stable dimensional
  complexity $n$ and $[X] = [Y]$ then they are stably birational.
\end{corollary}

\begin{proof}
  Let $n = \dim X$.  If $X$ and $Y$ are birational then we are done.  If they
  are not birational then in the spectral sequence for $K(\V)$ we have $[X] \neq
  [Y]$ in $E^1_{0,n}$, but $[X] = [Y]$ in $E^\infty_{0,n}$.  Let $r$ be the
  minimal integer such that $[X] = [Y]$ in $E^r_{0,n}$.  Projecting down to
  $\tilde E^r_{0,n}$ we see that $[X] = [Y]$ there as well.  If $[X] \neq 0$
  then by Lemma~\ref{lem:nice-dcomp}(4) we must have $X \sim_{r-1} Y$, and $X$
  and $Y$ are stably birational, as desired.  On the other hand, if $[X] = 0$
  then by Lemma~\ref{lem:nice-dcomp}(4) there exists a birational isomorphism
  class $\alpha \in \ell(B_{n-1})$ such that $X\times \A^{r-1}\in
  \ell^{r-1}(\alpha)$; in this case $X$ has stable dimensional complexity less
  than $n$, contradicting the assumption in the corollary.  Thus this cannot
  happen and $X$ and $Y$ are stably birational, as desired.
\end{proof}

Note that this does not contradict Borisov's construction in \cite{borisov14}:
in that paper, Borisov constructs two varieties $X$ and $Y$ which are not stably
birational but which have $[X] = [Y]$.  However, the fact that $[X] = [Y]$ is
shown by having $[X] = [X'] \LL^7$ and $[Y] = [Y']\LL^7$, with no consideration
of whether $7$ is the smallest power of $\LL$ necessary for the result.  Borisov
shows that the power of $\LL$ must be at least $1$ (and therefore that
multiplication by $\LL$ has a kernel); this corollary is another proof of that
fact, since if $[X'] = [Y']$ then $X'$ and $Y'$ (and therefore $X$ and $Y$)
would have to be stably birational.

We are now ready to prove Theorem~\ref{thm:ll} over a convenient field.  


\begin{proof}[Proof of Theorem~\ref{thm:ll}]
  By Lemma~\ref{lem:nice-dcomp}(4) we know that 
  $\tilde E^\infty_{0,n} \cong \Z^{\oplus (B_n/\sim_\infty) \bs \ell(B_{n-1})}$.
  Since the groups in the $0$ column are free, this means that 
  \[K_0(\V)/(\LL) \cong K_0(\V/L) \cong \bigoplus_{n \geq 0} \tilde E^\infty_{0,n} 
  \cong \bigoplus_{n \geq 0} \Z^{\oplus (B_n/\sim_\infty) \bs \ell(B_{n-1})}.\]
  We claim that this contains one copy of $\Z$ for each stable birational
  isomorphism class.  

  From the analysis in Lemma~\ref{lem:nice-dcomp} this contains at least one
  copy of $\Z$ for each stable birational isomorphism class, since there is at
  least one at $\tilde E^1$ and the differentials only quotient out differences
  between stably birational varieties.  Thus we just need to check that at
  $E^\infty$ only one representative of each class is left.  From the formula in
  part (4) we have quotiented out by stable birational equivalence in each
  degree, so we just need to check that we also quotient out the difference in
  different degrees.

  Let $\eta$ be a stable birational isomorphism class, and let $X$ be a
  representative of this stable birational isomorphism class of minimal
  dimension; write $ m = \dim X$.  We claim that if $\alpha$ in $B_n$ satisfies
  $\ell^{n-m+r}([X]) = \ell^r(\alpha)$ then $[\alpha] = 0$ in $\tilde E^\infty_{0,n}$.
  For convenience, let us suppose that $r$ is minimal.  Then inside
  $\tilde E^1_{1,n-m+r}$ we have a nonzero element represented by
  $[\ell^{n-m+r-1}([X])] - [\ell^{r-1}(\alpha)]$, and
  \[d_{r-1}([\ell^{(n-m+r-1)}([X])] - [\ell^{r-1}(\alpha)]) =
  [\ell^{n-m}([X])] - [\alpha] = -[\alpha] \in \tilde E^{r-1}_{0,n},\]
  since $m<n$.  Thus $[\alpha] = 0$ in $\tilde E^r_{0,n}$ and therefore also in
  $\tilde E^\infty_{0,n}$.  The only representative of the stable birational
  isomorphism class $\eta$ that survives to $\tilde E^\infty$ is $[X]$.
\end{proof}

In addition, from this proof we see that every generator in $E^\infty_{0,n}$ is
a representative of a stable birational isomorphism class whose minimal
representative appears in dimension $n$.  Thus the spectral sequence for
$K(\V/L)$ keeps track not only of ``minimal stability degree'', but also
of the ``minimal dimension'' of a stable birational isomorphism class.


\begin{proof}[Proof of Theorem~\ref{thm:zero=>notinj}]
  We prove the contrapositive of this theorem: if all $\psi_n$ have trivial
  kernels then $\LL$ is not a zero divisor.

  The cofiber sequence 
  $K(\V) \rto^{K(L)} K(\V) \rto K(\V/L)$
  gives rise to a long exact sequence in homotopy
  \[\makeshort{K_1(\V) \rto K_1(\V/L) \rto K_0(\V) \rto^{\cdot \LL}
    K_0(\V).}\]
  Thus the kernel of multiplication by $\LL$ is equal to the cokernel of the map
  $K_1(\V) \rto K_1(\V/L)$.  We therefore want to analyze the cokernel of the
  map $E^\infty_{1,q} \rto \tilde E^\infty_{1,q}$.  From the computation of the
  differentials in Lemma~\ref{lem:nice-dcomp} we see that 
  \[\tilde E^\infty_{1,n} \cong \left[\bigg(\bigoplus_{\beta\in B_n} \pi_1\tilde
    C_\beta\bigg) \oplus \Z/2^{\oplus B_n \bs \ell(B_{n-1})}\right],\] where the
  square brackets denote some quotient of the given group.  By
  Corollary~\ref{cor:1-surj}  $E^1_{1,n} \rto \tilde E^1_{1,n}$ is
  surjective on this component.

  Suppose that all birational automorphisms of varieties extend to piecewise
  automorphisms.  Then by Theorem~\ref{thm:ll<=>diffs}  all
  differentials between the $1$-st and $0$-th columns of the spectral sequence
  are $0$.  Since the map between spectral sequences is surjective on all pages
  on the component of the $E^1$ page that survives to $E^\infty$, and all
  differentials out of this component are zero, the map on $E^\infty$ pages is
  also surjective, and therefore the map $K_1(\V) \rto K_1(\V/L)$ is surjective;
  thus $\cdot \LL$ is injective and $\LL$ is not a zero divisor.
\end{proof}

By doing a little bit more spectral sequence analysis, we can finally prove
Theorem~\ref{thm:kernelL}. 


\begin{proof}[Proof of Theorem~\ref{thm:kernelL}]
  Pick a representative $x$ of minimal filtration degree $n$ in the preimage of
  $\chi$ in $K_1(\V/L)$.  Then in the long exact sequence
  \[\makeshort{K_1(\V^{(n)}) \rto K_1(\V^{(n)}/L) \rto^\partial K_0(\V^{(n-1)}) \rto^{\cdot
      \LL} K_0(\V^{(n)})},\]
  write $\partial[x] = [X]-[Y]$, with $X$ and $Y$ of minimal dimension.  We
  know that $\dim X = \dim Y$ because $[X\times \A^1] = [Y\times \A^1]$ and this
  means that $\dim X + 1 = \dim Y+1$ by \cite[Corollary 5]{liusebag}.   We claim
  that $\dim X < n-1$. Indeed, consider the commutative diagram:
  \begin{diagram}
    { & K_1(\V^{(n)}/\iota) \\
      K_1(\V^{(n)}/L) & K_1((\V^{(n)}/L)/\tilde\iota) \\ 
      K_0(\V^{(n-1)}) & K_0(\V^{(n-1)}/\iota) \\};
      \to{2-1}{2-2}^{pr_h} \to{1-2}{2-2}^{pr_v} \to{2-2}{3-2}^\partial
      \to{2-1}{3-1}^\partial \to{3-1}{3-2}^{pr'_h}
  \end{diagram}
  where the two columns are obtained from the long exact sequences in homotopy
  of the two right-hand columns of Diagram~\ref{diag:total-cofiber} and the
  horizontal morphisms are induced by the morphisms in the rows of the diagram.
  From our choice of $n$, $x$ is in $K_1(\V^{(n)}/L)$ and $pr_h(x) \neq 0$; since $x$
  is a permanent cycle in the spectral sequence for $K(\V/L)$, by
  Corollary~\ref{cor:1-surj} $pr_h(x)$ is in the image of $pr_v$; thus
  $\partial(pr_h(x)) = 0$. On the other hand,
  $\partial(pr_h(x)) = pr'_h(\partial(x)) = (pr'_h)([X]-[Y])$.

  Since $K_0(\V^{(n-1)}/\iota)$ is the free abelian group on $B_{n-1}$ and
  $pr'_h([X]-[Y]) = 0$ this means that the components of $X$ and $Y$ of
  dimension $n-1$ are pairwise birational.  In particular, since $X$ and $Y$
  were chosen to have minimal dimension, this means that $\dim X < n-1$.

  Now consider the commutative diagram
  \begin{diagram}
    { K_1(\V^{(n-1)}) & K_1(\V^{(n-1)}/L) & K_0(\V^{(n-2)}) &
      K_0(\V^{(n-1)}) \\
      K_1(\V^{(n)}) & K_1(\V^{(n)}/L) & K_0(\V^{(n-1)}) & K_0(\V^{(n)}) \\};
    \to{1-1}{1-2} \to{1-2}{1-3}^\partial \to{1-3}{1-4}^{\cdot \LL} \to{2-1}{2-2}
    \to{2-2}{2-3}^\partial \to{2-3}{2-4}^{\cdot \LL} \to{1-1}{2-1} \to{1-2}{2-2}
    \to{1-3}{2-3}^{K_0(\iota)} \to{1-4}{2-4}^{K_0(\iota)}
  \end{diagram}
  where the top and bottom rows are exact.  Then $[X]-[Y]$ lives in
  $K_0(\V^{(n-2)})$, but it is not in the image of $\partial$, since the first
  filtration degree that contains $x$ is $n$.  Thus $\LL([X]-[Y]) \neq 0$.
  However, $K_0(\iota)([X]-[Y])$ is in the image of $\partial$, so
  $\LL(K_0(\iota)([X]-[Y])) = 0$; thus $\LL([X]-[Y])$ is in the kernel of
  $K_0(\iota)$, and thus of $\psi_n$.  Therefore $[X\times \A^1] = [Y\times
  \A^1]$ in $K_0(\V)$ but they are not piecewise isomorphic.
\end{proof}

We conclude this paper with the following conjecture, which attempts to correct
Question~\ref{q:ll}.

\begin{conjecture}
  Suppose that $X$ and $Y$ are varieties over a convenient field $k$ such that
  $[X] = [Y]$ in $K_0(\V_k)$.  Then there exist varieties $X'$ and $Y'$ such
  that $[X'] \neq [Y']$,
  $[X'\times \A^1] = [Y'\times \A^1]$, and $X \amalg
  (X'\times \A^1)$ is piecewise isomorphic to $Y \amalg (Y'\times \A^1)$.
\end{conjecture}

This conjecture says that elements in the kernel of multiplication by $\LL$ are
the only possible ``errors'' to equality in the Grothendieck ring implying that
varieties are piecewise isomorphic.  Note that the definition of a convenient
field is the statement that this is the case for inclusions $K_0(\V^{(n)}) \rto
K_0(\V^{(n+1)})$.  In addition, Corollary~\ref{cor:ls} states that if $[X]=[Y]$
for certain $X$ and $Y$ then $X$ and $Y$ are stably birational: that the error
to their being birational to begin with is a power of $\LL$.

\bibliographystyle{IZ}
\bibliography{IZ-all}
\end{document}